\documentclass{amsart}
\newtheorem{theor}{Theorem}[section]
\theoremstyle{definition} 

\theoremstyle{remark} \newtheorem{rem}{Remark}[section]
\newcommand{\pn}{\par\noindent} \newcommand{\pmn}{\par\medskip\noindent}

\begin{document}
\title{An associative Latin square is a group. A very short proof}
\author{Yury Kochetkov}\footnote{Department of Applied Mathematics,
Higher School of Economics, Moscow, Russia}
\date{}
\keywords{}
\begin{abstract} A Latin square of order $n$ with symbols
$a_1,\ldots,a_n$ can be considered as a multiplication table for
binary operation in the set $A=\{a_1,\ldots,a_n\}$. We prove that,
if this operation is associative, then $A$ is a group.
\end{abstract}

\email{yukochetkov@hse.ru, yuyukochetkov@gmail.com} \maketitle

\section{A short introduction} \pn A Latin square of order $n$ with symbols
$a_1,\ldots,a_n$ is a $n\times n$ table, where each row and each
column is a permutation of these symbols \cite{DK,W}. The table
below is a Latin square of order 4 with symbols $\{1,2,3,4\}$
\[\begin{tabular}{|c|c|c|c|} \hline 2&3&1&4\\ \hline 1&4&2&3\\ \hline
3&1&4&2\\ \hline 4&2&3&1\\ \hline \end{tabular}\] A Latin can be
considered as the multiplication table of a binary operation
defined in the set $A=\{a_1,\ldots,a_n\}$. In what follows, words
"an associative Latin square", "a Latin square with unit" and so
on, mean properties of the corresponding binary operation. \pmn
The multiplication table of a finite group is a Latin square
\cite{S}. So we can formulate a natural problem: "When a Latin
square is the multiplication table of some group?". It turns out
that the associativity is a sufficient property here, i.e., if a
Latin square is associative, then the corresponding binary
operation has a unit and each element of $A$ has an inverse. \pmn
\begin{rem} The author did not find this result in literature.
\end{rem}

\section{Theorem and its proof}
\pn \begin{theor} An associative Latin square is the
multiplication table of some group. \end{theor}
\begin{proof} \underline{The first step.} Let $Q$ be our Latin square. At
first we want to prove that $Q$ contains an idempotent, i.e. such element
$a$ that $a\cdot a=a$. Let us assume that there are no idempotents in $Q$,
i.e. $a_i\cdot a_i\neq a_i$, $i=1,\ldots,n$. As the first row of $Q$ contains
$a_1$, then $a_1$ is some position $k$, $k>1$. It means that
$a_1\cdot a_k=a_1$. As $a_k$ is not an idempotent then $a_k\cdot a_k=a_l$,
$l\neq k$. Then
$$a_1=a_1\cdot a_k=(a_1\cdot a_k)\cdot a_k=a_1\cdot (a_k\cdot a_k)=
a_1\cdot a_l.$$ And we have a contradiction, because the first row
contains $a_1$ in two positions: the $k$-th and the $l$-th. \pmn
\underline{The second step.} Thus, $a_k$ is an idempotent.
Actually, $a_k$ is a unit. Let us assume that it is wrong and
$a_i\cdot a_k=a_j$, $i\neq j$, for some $i$. We have
$$a_j=a_i\cdot a_k=a_i\cdot (a_k\cdot a_k)=(a_i\cdot a_k)\cdot a_k=
a_j\cdot a_k.$$ It means that the $k$-th column contains two elements $a_j$:
in the $i$-th and in the $j$-th positions. Now let $a_k\cdot a_i=a_j$,
$i\neq j$, for some $i$. We have
$$a_j=a_k\cdot a_i=(a_k\cdot a_k)\cdot a_i=a_k\cdot (a_k\cdot a_i)=
a_k\cdot a_j.$$ It means that the $k$-th row contains two elements
$a_j$: in the $i$-th position and in the $j$-th position. \pmn
\underline{The third step.} Thus, $a_k$ is a unit. Now we must
prove that each element $a_i$ has an inverse. It is easy. The
$i$-th row and the $i$-th column contains element $a_k$ (in the
$l$-th and in the $m$-th position, respectively). It means that
$a_i\cdot a_l=a_k$ and $a_m\cdot a_i=a_k$. But
$$a_l=a_k\cdot a_l=(a_m\cdot a_i)\cdot a_l=a_m\cdot (a_i\cdot a_l)=
a_m\cdot a_k=a_m$$. \end{proof} \begin{rem} There exists a
non-associative Latin square with a unit and inverse elements, for
example
\[\begin{tabular}{|c|c|c|c|c|c|} \hline 1&2&3&4&5&6\\ \hline 2&3&1&5&6&4\\
\hline 3&1&2&6&4&5\\ \hline 4&6&6&1&2&3\\ \hline 5&4&6&2&3&1\\
\hline 6&5&4&3&1&2\\ \hline  \end{tabular}\] here
$$3=6\cdot 4=(4\cdot 2)\cdot 4\neq 4\cdot (2\cdot 4)=4\cdot 5=2.$$
\end{rem}

\vspace{5mm}

\end{document}